\documentclass[12pt]{article}
\usepackage{amsthm}

\usepackage[latin1]{inputenc}
\usepackage{subfigure}
\usepackage{color}
\usepackage{amsmath}
\usepackage{amsthm}
\usepackage{amstext}
\usepackage{amssymb}
\usepackage{amsfonts}
\usepackage{graphicx}
\usepackage{mathrsfs}
\usepackage{bbm}
\usepackage[all]{xy}
\usepackage{mathabx}
\usepackage{tikz}
\usepackage{mathtools}
\usepackage{tabularx}
\usepackage{array}
\usepackage{enumitem}
\usepackage{listings}
\usepackage{url}
\DeclareMathOperator{\Tr}{Tr}

\theoremstyle{plain}

\newtheorem{theorem}{Theorem}[section]
\newtheorem{conjecture}[theorem]{Conjecture}

\newtheorem{lemma}[theorem]{Lemma}

\theoremstyle{definition}

\DeclareMathAlphabet{\mathpzc}{OT1}{pzc}{m}{it}

\usepackage{amssymb,amsmath,tabularx,graphicx}
\usepackage{tikz}
\usetikzlibrary[calc,intersections,through,backgrounds,arrows,decorations.pathmorphing]

\begin{document}
	
	\title{Induced subgraphs of hypercubes and \\a proof of the Sensitivity Conjecture}
	\author{Hao Huang \thanks{Department of Mathematics, Emory University, Atlanta, USA. Email: hao.huang@emory.edu. Research supported in part by the Collaboration Grants from the Simons Foundation.}}
	\date{}

\maketitle
\abstract{In this paper, we show that every $(2^{n-1}+1)$-vertex induced subgraph of the $n$-dimensional cube graph has maximum degree at least $\sqrt{n}$. This result is best possible, and  improves a logarithmic lower bound shown by Chung, F\"uredi, Graham and Seymour in 1988. As a direct consequence, we prove that the sensitivity and degree of a boolean function are  polynomially related, solving an outstanding foundational problem in theoretical computer science, the Sensitivity Conjecture of Nisan and Szegedy.
}
\section{Introduction}
Let $Q^n$ be the $n$-dimensional hypercube graph, whose vertex set consists of vectors in $\{0, 1\}^n$, and two vectors are adjacent if they differ in exactly one coordinate. For an undirected graph $G$, we use the standard graph-theoretic notations $\Delta(G)$ for its maximum degree, and $\lambda_1(G)$ for the largest eigenvalue of its adjacency matrix. In 1988, Chung, F\"uredi, Graham, and Seymour \cite{cfgs_cube} proved that if $H$ is an induced subgraph of more than $2^{n-1}$ vertices of $Q^n$, then the maximum degree of $H$ is at least $(1/2-o(1)) \log_2 n$. Moreover, they constructed a $(2^{n-1}+1)$-vertex induced subgraph whose maximum degree is $\lceil \sqrt{n}\, \rceil$.

In this short paper, we prove the following result, establishing a sharp lower bound that matches their construction. \textcolor{black}{Note that the $2^{n-1}$ even vertices of $Q^n$ induce an empty subgraph. This theorem shows that any subgraph with just one more vertex would have its maximum degree suddenly jump to $\sqrt{n}$.} 
\begin{theorem}\label{thm_main}
For every integer $n \ge 1$, let $H$ be an arbitrary $(2^{n-1}+1)$-vertex induced subgraph of $Q^n$, then 
$$\Delta(H) \ge \sqrt{n}.$$
Moreover this inequality is tight when $n$ is a perfect square.
\end{theorem}
The induced subgraph problem is closely related to one of the most important and challenging open problems in theoretical computer science: the Sensitivity vs. Block Sensitivity Problem. In his 1989 paper, Nisan \cite{nisan} gave right bounds for computing the value of a boolean function in the CREW-PRAM model. These bounds are expressed in terms of two complexity measures of boolean functions. For $x \in \{0, 1\}^n$ and a subset $S$ of indices from $[n]=\{1, \cdots, n\}$, we denote by $x^S$ the binary vector obtained from $x$ by flipping all indices in $S$. For $f: \{0, 1\}^n \rightarrow \{0, 1\}$, the {\it local sensitivity} $s(f, x)$ on the input $x$ is defined as the number of indices $i$, such that $f(x) \neq f(x^{\{i\}})$, and the {\it sensitivity} $s(f)$ of $f$ is $\max_{x} s(f, x)$. \textcolor{black}{The sensitivity measures the local changing behavior of a boolean function with respect to the Hamming distance. It can be viewed as a discrete analog of the smoothness of continuous functions (see \cite{GNSTW} for more in-depth discussions).} The {\it local block sensitivity} $bs(f, x)$ is the maximum number of disjoint blocks $B_1, \cdots, B_k$ of $[n]$, such that for each $B_i$, $f(x) \neq f(x^{B_i})$. Similarly, the {\it block sensitivity} $bs(f)$ of $f$ is $\max_{x} bs(f, x)$. Obviously $bs(f) \ge s(f)$. A major open problem in complexity theory was posed by Nisan and Szegedy \cite{nisan-szegedy}, asking whether they are polynomially related.
\begin{conjecture} (Sensitivity Conjecture) 
There exists an absolute constant $C>0$, such that for every boolean function $f$, 
$$bs(f) \le s(f)^C.$$
\end{conjecture}

Although seemingly unnatural, the block sensitivity is known to be polynomially related to many other important complexity measures of boolean functions, including the decision tree complexity, the certificate complexity, the quantum and randomized query complexity, and the degree of the boolean function (as real polynomials), and the approximate degree \cite{survey}. \textcolor{black}{It is noteworthy that some of these relationship is quite subtle. For instance, although the degree and approximate degree both concern algebraic properties of boolean functions, the only known proof of their polynomial relationship goes through other more combinatorial notions.}

The Sensitivity Conjecture, if true, would place the sensitivity in the same category with the other complexity measures listed above. \textcolor{black}{Computationally, it would imply that ``smooth'' (low-sensitivity) functions are easy to compute in some of the simplest models like the deterministic decision tree model. Algebraically, it asserts that such functions have low degree as real polynomials. Combinatorially, as observed by Gotsman and Linial \cite{gotsman-linial}, it is equivalent to the previous cube problem. We will discuss this connection later.}

Despite numerous attempts for almost thirty years, the Sensitivity conjecture still remains wide open, and the best upper bound of $bs(f)$ is exponential in terms of $s(f)$. For example, Kenyon and Kutin \cite{kenyon-kutin} showed that $bs(f)=O(e^{s(f)}\sqrt{s(f)})$. For the lower bound, Rubinstein \cite{rubinstein} first proposed a boolean function $f$ with $bs(f)=\frac{1}{2}s(f)^2$, showing a quadratic separation between these two complexity measures. Virza \cite{virza}, and subsequently Ambainis and Sun \cite{ambainis-sun} obtained better constructions which still provides quadratic separations.  For a comprehensive survey with more background and discussions, in particular the many problems equivalent to the Sensitivity Conjecture, we refer the readers to the surveys of Buhrman and de Wolf \cite{survey_bw},  Hatami, Kulkarni and Pankratov \cite{survey}, and some recent works \cite{drucker,GKS,GSW,tal}.

Recall that $Q^n$ denotes the $n$-dimensional cube graph. For an induced graph $H$ of $Q^n$, let $Q^n-H$ denote the subgraph of $Q^n$ induced on the vertex set $V(Q^n) \setminus V(H)$. Let $\Gamma(H)=\max\{\Delta(H), \Delta(Q^n-H)\}$. The {\it degree} of a boolean function $f$, denoted by $\deg(f)$, is the degree of the unique multilinear real polynomial that represents $f$. Gotsman and Linial \cite{gotsman-linial} proved the following  remarkable equivalence using Fourier analysis.

\begin{theorem} (Gotsman and Linial \cite{gotsman-linial})
The following are equivalent for any monotone function $h: \mathbb{N} \rightarrow \mathbb{R}$.\\
(a) For any induced subgraph $H$ of $Q^n$ with $|V(H)|\neq 2^{n-1}$, we have $\Gamma(H) \ge h(n)$.\\
(b) For any boolean function $f$, we have $s(f) \ge h(\deg(f))$. 
\end{theorem}

Note that Theorem \ref{thm_main} implies that $h(n)$ can be taken as $\sqrt{n}$, since one of $H$ and $Q^n-H$ must contain at least $2^{n-1}+1$ vertices, and the maximum degree $\Delta$ is monotone. As a corollary, we have
\begin{theorem}
For every boolean function $f$,
$$s(f) \ge \sqrt{\deg(f)}.$$
\end{theorem}
This confirms a conjecture of Gotsman and Linial \cite{gotsman-linial}. This inequality is also tight for the AND-of-ORs boolean function \cite[Example 5.2]{survey}. Recall that the degree and the block sensitivity are polynomially related. Nisan and Szegedy \cite{nisan-szegedy} showed that $bs(f) \le 2 \deg(f)^2$ and this bound was later improved by Tal \cite{tal} to $bs(f) \le \deg(f)^2$. Combining these results we have confirmed the Sensitivity Conjecture.

\begin{theorem}\label{thm_sensitivity}
For every boolean function $f$, 
$$bs(f) \le s(f)^4.$$
\end{theorem}

\section{Proof of the main theorem}
To establish Theorem \ref{thm_main}, we prove a series of lemmas. Given a $n \times n$ matrix $A$, a {\it principal submatrix} of $A$ is obtained by deleting the same set of rows and columns from $A$.
\begin{lemma}\label{lem_cauchy} (Cauchy's Interlace Theorem)
Let $A$ be a symmetric $n \times n$ matrix, and $B$ be a $m \times m$ principal submatrix of $A$, for some $m<n$. If the eigenvalues of $A$ are $\lambda_1 \ge \lambda_2 \ge \cdots \ge \lambda_n$, and the eigenvalues of $B$ are $\mu_1 \ge \mu_2 \ge \cdots \ge \mu_m$, then for all $1 \le i \le m$,
$$\lambda_i \ge \mu_i \ge \lambda_{i+n-m}.$$ 
\end{lemma}
Cauchy's Interlace Theorem is a direct consequence of the Courant-Fischer-Weyl min-max principle. A direct proof can also be found in \cite{fisk}. 

\begin{lemma}\label{lem_eig}
We define a sequence of symmetric square matrices iteratively as follows,
$$A_1=\begin{bmatrix} 0 & 1 \\
1 & 0 \end{bmatrix},~~A_{n}=\begin{bmatrix} A_{n-1} & I \\
I & -A_{n-1} \end{bmatrix}.$$
Then $A_n$ is a $2^{n} \times 2^n$ matrix whose eigenvalues are $\sqrt{n}$ of multiplicity $2^{n-1}$, and $-\sqrt{n}$ of multiplicity $2^{n-1}$.
\end{lemma}
\begin{proof}
We prove by induction that $A_n^2=n I$. For $n=1$, $A_1^2=I$. Suppose the statement holds for $n-1$, that is $A_{n-1}^2=(n-1)I$, then 
$$A_n^2=\begin{bmatrix} A_{n-1}^2+I & 0 \\
0 & A_{n-1}^2+I\end{bmatrix}=nI.$$
Therefore, the eigenvalues of $A_n$ are either $\sqrt{n}$ or $-\sqrt{n}$. Since $\Tr[A_n]=0$, we know that $A_n$ has exactly half of the eigenvalues being $\sqrt{n}$ and the rest being $-\sqrt{n}$.
\end{proof}

\begin{lemma}\label{lem_key}
Suppose $H$ is an $m$-vertex undirected graph, and $A$ is a symmetric matrix whose entries are in $\{-1, 0, 1\}$ and whose rows and columns are indexed by $V(H)$, and whenever $u$ and $v$ are non-adjacent in $H$, $A_{u,v}=0$. Then 
$$\Delta(H) \ge \lambda_1:=\lambda_1(A).$$  
\end{lemma}
\begin{proof}
Suppose $\vec{v}$ is the eigenvector corresponding to $\lambda_1$. Then $\lambda_1 \vec{v} = A \vec{v}.$
Without loss of generality, assume $v_1$ is the coordinate of $\vec{v}$ that has the largest absolute value. Then
$$|\lambda_1v_1| = |(A \vec{v})_1| = \left|\sum_{j=1}^m A_{1, j} v_j\right| = \left|\sum_{j \sim 1} A_{1, j} v_j \right| \le \sum_{j \sim 1} |A_{1, j}||v_1| \le \Delta(H) |v_1|.$$
Therefore $|\lambda_1| \le \Delta(H)$.
\end{proof}
With the lemmas above, we are ready to prove the main theorem.
\begin{proof}[Proof of Theorem \ref{thm_main}]
Let $A_n$ be the sequence of matrices defined in Lemma \ref{lem_eig}. Note that the entries of $A_n$ are in $\{-1,0,1\}$. By the iterative construction of $A_n$, it is not hard to see that when changing every $(-1)$-entry of $A_n$ to $1$, we get exactly the adjacency matrix of $Q^n$, and thus $A_n$ and $Q^n$ satisfy the conditions in Lemma \ref{lem_key}. For example, we may let the upper-left and lower-right blocks of $A_n$ correspond to the two $(n-1)$-dimensional subcubes of $Q^n$, and the two identity blocks correspond to the perfect matching connecting these two subcubes. Therefore, a $(2^{n-1}+1)$-vertex induced subgraph $H$ of $Q^n$ and the principal submatrix $A_H$ of $A_{n}$ naturally induced by $H$ also satisfy the conditions of Lemma \ref{lem_key}. As a result,
$$\Delta(H) \ge \lambda_1(A_H).$$
On the other hand, from Lemma \ref{lem_eig}, the eigenvalues of $A_n$ are known to be 
$$\sqrt{n}, \cdots, \sqrt{n}, -\sqrt{n}, \cdots, -\sqrt{n}.$$
Note that $A_H$ is a $(2^{n-1}+1) \times (2^{n-1}+1)$ submatrix of the $2^n \times 2^n$ matrix $A_n$. By Cauchy's Interlace Theorem,
$$\lambda_1(A_H) \ge \lambda_{2^{n-1}}(A_n)=\sqrt{n}.$$
Combining the two inequalities we just obtained, we have
$\Delta(H) \ge \sqrt{n},$ completing the proof of our theorem.
\end{proof}

\noindent \textbf{Remark.} From the proof, one actually has $\lambda_1(H) \ge \lambda_1(A_H) \ge \sqrt{n}$. Since $\Delta(H) \ge \lambda_1(H)$, this result strengthens Theorem \ref{thm_main}. More interestingly, the inequality $\lambda_1(H) \ge \sqrt{n}$ is best possible for all $n$. This can be seen by taking all the even vertices and one odd vertex of $Q^n$, then the induced subgraph is a copy of the star $K_{1, n}$, together with many isolated vertices. The largest eigenvalue of this induced subgraph is exactly $\sqrt{n}$.

\section{Concluding Remarks}
In this paper we confirm the Sensitivity Conjecture by proving its combinatorial equivalent formulation discovered by Gotsman and Linial. The following problems might be interesting.

\begin{itemize}
\item Given a ``nice'' graph $G$ with high symmetry, denote by $\alpha(G)$ its independence number. Let $f(G)$ be the minimum of the maximum degree of an induced subgraph of $G$ on $\alpha(G)+1$ vertices. What can we say about $f(G)$? In particular, for which graphs, the method used in proving Theorem \ref{thm_main} would provide a tight bound? 

\item Back to the hypercube problem, let $g(n, k)$ be the minimum $t$, such that every $t$-vertex induced subgraph $H$ of $Q^n$ has maximum degree at least $k$. In this paper, we show that $g(n, \sqrt{n}) = 2^{n-1}+1$. It would be interesting to determine $g(n, k)$ asymptotically for other values of $k$.

\item Although we have shown a tight bound between the sensitivity and the degree, at the time of writing this paper, the best separation between the block sensitivity $bs(f)$ and the sensitivity $s(f)$ is $bs(f)=\frac{2}{3}s(f)^2-\frac{1}{3}s(f)$ shown in \cite{ambainis-sun}, which is quadratic. Theorem \ref{thm_sensitivity} only shows a quartic upper bound. Perhaps one could close this gap by directly applying the spectral method to boolean functions instead of to the hypercubes.
\end{itemize}

\noindent \textbf{Acknowledgment. }The author would like to thank Benny Sudakov for simplifying the proof of Lemma \ref{lem_eig}, Jacob Fox, Yan Wang and Yao Yao for reading an earlier draft of the paper, Matthew Kwan, Po-Shen Loh and Jie Ma for the inspiring discussions over the years, and the anonymous referee for helpful suggestions on improving the presentation of this paper.


\begin{thebibliography}{99}
\bibitem{ambainis-sun}
A.~Ambainis, X.~Sun,
\newblock{New separation between $s(f)$ and $bs(f)$},
\newblock{\em Electronic Colloquium on Computational Complexity (ECCC)}, {\bf 18}, 116, 2011. Available at \url{http://eccc.hpi-web.de/report/2011/11}.
	

\bibitem{survey_bw}
H.~Buhrman, R.~de Wolf,
\newblock{Complexity measures and decision tree complexity: a survey},
\newblock{\em Theoretical Computer Science} {\bf 288} (2002), 21--43.

\bibitem{cfgs_cube}
F.~Chung, Z.~F\"uredi, R.~Graham, P.~Seymour,
\newblock{On induced subgraphs of the cube},
\newblock{\em J. Comb. Theory, Ser. A},
{\bf 49} (1) (1988), 180--187.

\bibitem{drucker}
A.~Drucker,
\newblock{A note on a communication game}, 
CoRR, abs/1706.07890, 2017.

\bibitem{fisk}
S.~Fisk,
\newblock{A very short proof of Cauchy's interlace theorem for eigenvalues of Hermitian matrices},
\newblock{\em Amer. Math. Monthly} {\bf 112} (2005), no. 2, 118. 

\bibitem{GKS}
J.~Gilmer, M.~Kouck\'y, M.~Saks,
\newblock{A new approach to the sensitivity conjecture},
\newblock{\em Proceedings of the 2015 Conference on Innovations in Theoretical Computer Science}, ACM, 2015, 247--254.

\bibitem{GNSTW}
P.~Gopalan, N.~Nisan, R.~Servedio, K.~Talwar, and A.~Wigderson,
\newblock{Smooth Boolean functions are easy: efficient algorithms for low-sensitivity functions}, 7th Innovations in Theoretical Computer Science Conference (ITCS), 2016. 

\bibitem{GSW}
P.~Gopalan, R.~Servedio, A.~Tal, A.~Wigderson,
\newblock{Degree and sensitivity: tails of two distributions},
\newblock{\em 31st Conference on Computational Complexity}, CCC 2016, Tokyo, Japan, 13:1--13:23.
\bibitem{gotsman-linial}
C.~Gotsman, N.~Linial,
\newblock{The equivalence of two problems on the cube.}
\newblock{\em J. Combin. Theory Ser. A}, {\bf 61} (1) (1992), 142--146.

\bibitem{survey}
P.~Hatami, R.~Kulkarni, D.~Pankratov, 
\newblock{Variations on the Sensitivity Conjecture},
\newblock{\em Theory of Computing Library}, Graduate Surveys {\bf 4} (2011), 1--27.

\bibitem{kenyon-kutin}
C.~Kenyon, S.~Kutin,   
\newblock{Sensitivity, block sensitivity, and $\ell$-block sensitivity of Boolean functions},
\newblock{\em Information and Computation}, {\bf 189} (1) (2004), 43--53.

\bibitem{nisan}
N.~Nisan,
\newblock{CREW PRAMs and decision trees}, 
\newblock{\em Proc. 21st STOC, New York, NY, ACM Press} (1989) 327--335.

\bibitem{nisan-szegedy}
N.~Nisan, M.~Szegedy,
\newblock{On the degree of Boolean functions as real polynomials},
\newblock{\em Comput. Complexity}, {\bf 4} (1992), 462--467.

\bibitem{rubinstein}
D.~Rubinstein,
\newblock{Sensitivity vs. block sensitivity of Boolean functions},
\newblock {\em Combinatorica}, {\bf 15} (2) (1995), 297--299.

\bibitem{tal}
A.~Tal,
\newblock{Properties and applications of boolean function composition},
\newblock{Proceedings of the 4th conference on Innovations in Theoretical Computer Science, ITCS '13}, 441--454.

\bibitem{virza}
M. Virza,
\newblock{Sensitivity versus block sensitivity of boolean functions}, 
\newblock{\em Information Processing Letters}, {\bf 111} (2011),  433--435.

\end{thebibliography}
\end{document}